\documentclass[twoside.11pt]%
{article}

\newsavebox{\savepar}

\makeatletter
\newcommand{\least}{\let\CS=\@currsize\renewcommand{\baselinestretch}{1}\tiny\CS}

\newcommand{\oneandahalfspacing}{\let\CS=\@currsize\renewcommand{\baselinestretch}{1.2}\tiny\CS}
\newcommand{\doublespacing}{\let\CS=\@currsize\renewcommand{\baselinestretch}{2.5}\tiny\CS}

\usepackage{amsmath}
\usepackage{amsfonts}
\usepackage{euscript}
\usepackage{oldgerm}
\usepackage{eufrak}
\usepackage{amsthm}
\usepackage{rotating}
\usepackage{dsfont}

   \renewcommand{\baselinestretch}{1.6}

\makeindex

  \newtheorem{theorem}{Theorem}[section]

  \theoremstyle{definition}

  \theoremstyle{remark}
  
  \newtheorem{prop}[theorem]{Proposition}

\theoremstyle{definition}

\theoremstyle{remark}

\def\liminf{\mathop{\underline{\hbox{lim}}}\limits} 
\def\limsup{\mathop{\overline{\hbox{lim}}}\limits}

\newcommand{\ba}{\begin{array}}
\newcommand{\ea}{\end{array}}

\newcommand{\NN}{\mathbb{N}}
\newcommand{\RR}{\mathbb{R}}

\DeclareMathOperator{\Span}{span}
\newcommand{\suchthat}{\;\ifnum\currentgrouptype=16 \middle\fi|\;}
\newcommand{\vect}{\mathrm{Vect}}
\title{Constrained Optimal Smoothing and Bayesian Estimation}
\author{X. Bay\footnotemark[2] \and L. Grammont\footnotemark[3]  
}

\footnotetext[2]{Mines Saint-\'Etienne, UMR   6158, LIMOS, F-42023 Saint-\'Etienne, France ({\tt bay@emse.fr}).}
\footnotetext[3]{Universit\'e de Lyon, Institut Camille Jordan, UMR 5208, 23 rue du Dr Paul Michelon, 
42023 Saint-\'Etienne Cedex 2, France.
( {\tt laurence.grammont@univ-st-etienne.fr}).}
 
\begin{document}
\sloppy
\maketitle

\noindent \textbf{Abstract}\\
In this paper, we extend the correspondence between Bayesian estimation and optimal smoothing in a Reproducing Kernel Hilbert Space (RKHS) adding a convexe constraints on the solution.   Through a sequence of approximating Hilbertian spaces and a discretized model, we prove that the Maximum A Posteriori (MAP)   of the posterior distribution is  exactly the optimal  constrained smoothing function in the RKHS. This paper can be read as a generalization of the paper \cite{wahba1} of Kimeldorf-Wahba where it is proved that the optimal smoothing solution is the mean of the posterior distribution.    
\bigskip

\noindent {\bf Keywords: correspondence; smoothing; inequality constraints; Reproducing Kernel Hilbert Space;   Bayesian estimation}  

\bigskip

\noindent
\section{Introduction}
  Consider $X$  a nonempty set of $\RR$ and $E$ a  set of   functions from $X$ to $\RR$. \\ 
 Given data $(x_i,y_i)\in X \times \RR $,  the   smoothing problem is   to find  a function $\widehat{u}$ minimizing 
\begin{equation}\label{smoothing}
  \|u\|_H^2+\displaystyle \frac{1}{\sigma^2} \sum_{i=1}^{n}(u(x_i)-y_i)^2
\end{equation} 
on an Hilbert space $H$ of 
 $E$.  \\
As Kimeldorf and Wahba explained it in \cite{wahba1}, the term  $\|u\|_H^2$ is the smoothness criterion for the solution and $\displaystyle\frac{1}{\sigma^2} \sum_{i=1}^{n}(u(x_i)-y_i)^2$ measures the disparity of $u$ with the data. $\widehat{u}$ is a compromise between smoothness and fidelity to the data. In \cite{wahba1}, the disparity of the data is measured by 
$\displaystyle \sum_{i=1}^{n}\sum_{j=1}^{n}(u(x_i)-y_i)b^{ij}(u(x_j)-y_j) $. We choose $b^{ij}=\displaystyle\frac{1}{\sigma^2}\delta_{ij}$ with no loss of generality,  to facilitate subsequent reading of the paper. In \cite{wahba1}, the authors consider $\|u\|_H^2:=\displaystyle\int_{-\infty}^{+\infty}(Lu)^2(t)dt$, where $L$ is a linear differential operator. In that case, under conditions,  the solution $\widehat{u}$ is an $L$-spline. In this paper, $H$ is any Reproducing Kernel Hilbert Space equipped with its associated norm.    \\

The notion of Reproducing Kernel Hilbert Space (RKHS) is the material with which one has  built a bridge from the determistic world of  optimization  to the probabilistic worlds of estimation.  Aronszajn \cite{aron1950} published the theory of reproducing kernel in 1950  and  Parzen\cite{parzen1959} published Statistical inference on time series by Hilbert space methods in 1959. Later Schwartz \cite{schwartz1964} extended its formalism to topological spaces. Through its covariance fonction, $K(x,x')=cov(U(x),U(x'))$, a Gaussian process $U$ is represented by the Hilbert space spanned by the kernel $K$.

\newpage
The aim of Kimeldorf and Wahba in \cite{wahba1} is to highlight the correspondence between the smoothing by spline and  Bayesian estimation. For that purpose, they consider a stochastic model in which the selection of the smoothing criterion corresponds to the specification of a prior distribution on $U$ which is a Centered Gaussian process with   covariance fonction $ K(s,t)$. At points $x_1,\ldots,x_n$ the random variables $Y_i=U(x_i)+{\cal E}_i$ are observed, where ${\cal E}=({\cal E}_i)_i$ is a centered Gaussian vector ${\cal N}(0,\sigma I)$, where $I$ is the identity matrix. \\
 They prove that $\widehat{u}$ solution of the  minimization of (\ref{smoothing}) is the  Bayesian estimation 
\begin{equation}\label{estimation}
\widehat{u}  = \mathds{E}[U(t)\vert Y_1=y_1,\ldots,Y_n=y_n], 
\end{equation}
where $\mathds{E}$  denotes expectation. In other terms, it means that we look for  $u$ in the Gaussian space $H=\overline{span\{ U_x, x\in X\}}$ associated to the Centered  Gaussian Process $(U_x)_{x \in X}$ whose covariance is $K$.  
 
In both framework, one can prove that the solution of the smoothing problem (\ref{smoothing}) or Bayesian estimation (\ref{estimation}) is
\begin{equation}\label{explicite}
\widehat{u}  (t)= \boldsymbol{y}  \left(\mathds{K}+\sigma^2 I\right)^{-1} \boldsymbol{k}(t)^{\top}   
\end{equation} 
where $\boldsymbol{k}(t)=\left( K\left(x_1,t\right),\ldots,K\left(x_n,t\right)\right)$, $\mathds{K}$ is the matrix 
$\left(K\left(x_i,x_j\right)\right)_{1\leq i,j\leq n}$ and $\boldsymbol{y}=(y_1,\ldots,y_n)$. \\
 
Now, consider that $u$ is known to satisfy some additional constraints given by $u \in C$ where $C$ is a closed convex set. 
In the present paper, we consider the  {\bf constrainted} smoothing problem of   finding  a function $\widehat{u}$, in $H$ and $C$,  minimizing 
\begin{equation}\label{constrained} 
  \|u\|_H^2+\displaystyle \frac{1}{\sigma^2} \sum_{i=1}^{n}(u(x_i)-y_i)^2,
\end{equation} 
over $H\cap C$. \\
   In \cite{Micchelli1}, Theorem 3.1,  Micchelli and Utreras proved that the solution exists and is unique under certain conditions. They also give the expression of the solution involving the projection on the convex set $C$, denoted by $P_C$, and  a nonlinear algebraic system expressed with $P_C$ and approximated usually by some Newton type methods (see \cite{dontchev1}).  As Andersson and Elfving wrote it in their paper \cite{Andersson1991}, to transform this result into a numerical algorithm, it is necessary to compute the orthogonal projection $P_C$ and the difficulty lies in that calculation. Andersson and Elfving   investigated the structure of the projection operator $P_C$  for a particular convex set $C$ representing monotonicity constraints.  \\
   
 The aim of the present paper is to rewrite the constrained smoothing problem  (\ref{constrained}) with a stochastic model so that the solution can be interpreted as a Bayesian estimation. We found out that it was possible through a discretization   of the constrained smoothing problem whose solution $\widehat{u}_N$  tends to  $\widehat{u}$. The integer $N$ is the discretization parameter. Section 3 and Section 4 are devoted to it. Then we define the equivalent finite dimensional approximation $U_N$ of the Gaussian process $U$ and we prove that the approximate solution $\widehat{u}_N$ can be interpreted as a Bayesian estimation. This estimation is  the MAP (Maximum A Posteriori) of the posterior distribution  of $U_N$ and not the mean like in the non constrainted smoothing problem. 
 
\section{Framework of   Constrained Optimal Smoothing} 
To simplify the paper, we suppose that $X=[0,1]$ and $E=C([0,1],\RR)$ is the linear   space
of real valued continuous functions on   $[0,1]$ equiped with the infinity norm. Let    
    $H$ be a RKHS of $E$ associated to the symmetric positive definite function $K$. Then, $H$ is an Hilbertian subspace of $E$ since
\begin{equation*}
\| h\|_E=\sup_{x\in X}|(h,K(.,x))_H|\leq c\|h\|_H,
\end{equation*}
where $c=\sup_{x\in X}K(x,x)^{1/2}<+\infty$.  \\
Let $(x_i,y_i)_{1\leq i \leq n}$  be the given data.\\
Let us define the function $J:H\longrightarrow \RR $ by 
\begin{equation}\label{J} 
J(u):=\|u\|_H^2+\displaystyle \frac{1}{\sigma^2} \sum_{i=1}^{n}(u(x_i)-y_i)^2
\end{equation}
The Constrained Optimal Smoothing (\ref{constrained}) can be rewritten as
\begin{equation}\label{prob}\tag{$P$}
  \quad \min_{u\in H\cap C } J(u)  
  \end{equation}
It is easy to see that $J$ is Fr\'echet differentiable and $\displaystyle\lim_{\Vert v\Vert \mapsto+\infty} J(v)=+\infty $. Moreover $J$ is strongly convex: $\forall u, v \in H, t \in [0,1],$
\begin{equation*}
J(tu+(1-t)v)\leq tJ(u)+(1-t)J(v) -t(1-t)\Vert u-v\Vert.
\end{equation*}
Then, if $C$ is a closed convex set of $E$ such that 
\begin{equation}\label{H1}\tag{$H1$}
H\cap C \neq \emptyset
\end{equation}
The problem $(P)$ has a unique solution, denoted by  $\widehat{u}  $.

\section{Discretization of Constrained Optimal Smoothing }

We propose a discretized optimization problem 
 (\ref{pbdis}) of (\ref{prob}) associated with   $\Delta_N$ a   subdivision of $[0,1]$  
\begin{eqnarray}\label{partition}
\Delta_N~: \quad 0=t_{0} < t_{1} < \ldots < t_{N}= 1, 
\end{eqnarray}
such that $\delta_N=\max\left\{|t_{i+1}-t_{i}|, \ i=0,\ldots,N-1 \right\} $ tends to zero as $N$ tends to infinity. We assume 
\begin{equation}\label{H2}\tag{$H2$}
\Delta_{N} \subset \Delta_{N+1} 
\end{equation}
and the data points $x_i$ belong to the partition 
\begin{equation}\label{H3}\tag{$H3$}
\{ x_1,\ldots,x_n\} \subset \Delta_{N} 
\end{equation}
 Let $H_N$ be 
 the classical subspace  of piecewise linear continuous functions associated to $\Delta_N$. A basis  of $H_N$ is   the so-called hat functions denoted by
$(\varphi_{0},\ldots,\varphi_{N})$.   
Next, we define  $\pi_N$ to be the classical piecewise linear interpolation  projection defined from $E$ onto $H_N$    by
$$\forall f\in E, \qquad \pi_N (f)=\sum\limits_{j=0}^{N} f(t_{j})\varphi_{j}.$$
We assume that 
\begin{equation}\label{H4}\tag{$H4$}
\pi_N(C)\subset C  
\end{equation}
According to a classical approximation result  
\begin{equation}\label{PN}
\pi_N(f) \underset{N\to +\infty}{\longrightarrow} f \quad \mbox{in $E$},
\end{equation}  
Let us define the linear evaluation operator ${\cal I}_N: E\mapsto \RR^{N+1}$ on the nodes $t_i$  and ${\cal I}_n:E\mapsto \RR^{n}$ on the data points $x_i$. 
$$\forall f \in E, \ {\cal I}_N(f):=\left(f(t_{0}),\ldots,f(t_{N})\right)^{\top} $$
$$\forall f \in E, \  {\cal I}_n(f):=\left(f(x_1),\ldots,f(x_{n})\right)^{\top}$$
To lighten the notations, $c_f$ denotes ${\cal I}_N(f)$. 
Let us define the matrix of $K$ on the nodes  
\begin{equation}\label{gamma_N}\Gamma_N=(K(t_i,t_j))_{0\leq i,j\leq N}
\end{equation} 
We suppose
 \begin{equation}\label{H5}\tag{$H5$} 
\Gamma_N \quad \mbox{is invertible}
\end{equation}
As $H$ is an RKHS associated to the kernel $K$, we can define a new  scalar product on $H_N$: $\forall u_N, v_N \in H_N$
\begin{equation}\label{proscalHN}
(u_N,v_N)_{H_N}:=c_{u_N}^{\top}\Gamma_N^{-1}c_{v_N} 
\end{equation}
which induces a norm on $H_N$: $\forall u_N  \in H_N $ 

\begin{equation}\label{normHN}
 \Vert u_N \Vert_{H_N}^2=c_{u_N}^{\top}\Gamma_N^{-1}c_{u_N}\end{equation} 
 Let us define the linear  operator $\rho_N: H_N \rightarrow H$ defined by 
\begin{equation}\label{QN}
 \forall v_N \in H_N, \ \rho_N(v_N):=\displaystyle\sum_{i=0}^{N} \lambda_i K(.,t_i) 
\end{equation}
where   $\Lambda=(\lambda_0,\ldots,\lambda_N)^{\top},$    solves  $$\Gamma_N\Lambda=c_{v_N}.$$ 
Let us notice that $\rho_N$ has been defined so that  $\rho_N\circ \pi_N$ is the orthogonal projection from $H$ onto $H_N^1$ where 
\begin{eqnarray}
H_1=\Span\left\{K(.,t_{j}), \ j=0,\ldots,N\right\}.
\end{eqnarray}
 The following proposition highlights the nature of the finite dimensional space $H_N$:

\begin{prop}\label{rkhsHN}   
 $H_N$ is a RKHS with kernel $K_N$ given by
\begin{equation}\label{KN}
\forall x', x\in [0,1], \qquad K_N(x',x)= \sum_{i,j=0}^{N} K(t_{i},t_{j}) \varphi_{j}(x)\varphi_{i}(x').
\end{equation}
For any $h_N\in H_N$, 
\begin{equation}\label{HN}
\|h_N\|_E\leq c\|h_N\|_{H_N},
\end{equation}
 where $c$ is a constant independent of $N$.  
\end{prop}

\begin{proof}
Clearly, $H_N$ is a finite-dimensional Hilbert space. Let $x$ be in $[0,1]$. We have 
\begin{equation*} 
K_N(.,x)=\sum_{i=0}^N \lambda_{i,x} \varphi_{i} \in H_N,
\end{equation*}
where $\lambda_{i,x}=\displaystyle\sum_{j=0}^N K(t_{i},t_{j})\varphi_{j}(x)=\left(\Gamma_N\varphi(x)\right)_i$, with 
\begin{equation}\label{phi}
\varphi(x):=(\varphi_{0}(x),\ldots,\varphi_{N}(x))^{\top}.
\end{equation} 
 Let $h:=\displaystyle\sum_{i=0}^N \alpha_{i} \varphi_{i}=\alpha^{\top}  \varphi(x) \in H_N$, \ $\alpha:=(\alpha_{0},\ldots,\alpha_{N} )^{\top}$.  We obtain 
\begin{eqnarray*}
\left(h, K_N(.,x)\right)_{H_N}=\alpha^{\top}  \Gamma_N^{-1}\left(\Gamma_N\varphi(x)\right)=\alpha^{\top}  \varphi(x)=h(x),
\end{eqnarray*}
which is the reproducing property in $H_N$. \\
For $x\in X$, we have
\begin{equation*}
|h(x)|=|(h,K_N(.,x))_{H_N}|\leq \|h\|_{H_N}\times \sqrt{K_N(x,x)},
\end{equation*}
where $K_N(x,x)=\sum_{i,j=0}^N K(t_{i},t_{j})\phi_{i}(x)\phi_{j}(x)$. Since $\sum_{i,j=0}^N\phi_{i}(x)\phi_{j}(x)=1$, we obtain
\begin{equation*}
0\leq \sup_{x\in X}K_N(x,x)\leq M=\max_{x,x'\in X}|K(x,x')|,
\end{equation*}
As $\|h_N\|_E=\|h_N\|_{\infty}$,  the proof of the lemma is completed.
 
\end{proof}

In the following proposition, one proves that the sequence of projections  $\pi_N$ is stable. \\    
It is straightforward that for all $f$ in $E$, 
\begin{equation*}
\| \pi_N(f)\|^2_{H_N}=c_{f}^{\top}\Gamma_N^{-1}c_{f},
\end{equation*}

\begin{prop}\label{stability} {\bf Stability of $\pi_N$} \\
 $\pi_N$ is stable, i.e.
\begin{equation}\label{stab}
\forall h\in H, \qquad \|\pi_N(h)\|_{H_N}\leq \|h\|_H.
\end{equation}
Moreover $H$ is characterized by  
\begin{equation}\label{newdefH}
H=\left\{h\in E~: \  \sup_{N} \|\pi_N(h)\|_{H_N}<+\infty\right\}
\end{equation}
and, for all $h\in H$, by 
\begin{equation}\label{normH}
\|f\|_H^2= \lim_{N\to +\infty}\|\pi_N(f)\|_{H_N}^2. 
\end{equation}
\end{prop}

\begin{proof}
Consider the usual orthogonal decomposition in the R.K.H.S $H$:
$H=H_0\overset{\perp}{\oplus} H_1$, with
\begin{eqnarray*}
H_0&=&\left\{h\in H \ : \ h(t_{j})=0, \ j=0,\ldots,N\right\},\\
H_1&=&\Span\left\{K(.,t_{j}), \ j=0,\ldots,N\right\}.
\end{eqnarray*}
For all $h\in H$, there exists a unique $h_0\in H_0$ and $h_1\in H_1$ such that $h=h_0+h_1$. Thus,
\begin{equation*}
\| h_1\|_{H}^2\leq \| h\|_{H}^2.
\end{equation*} 
Additionally, every $h_1\in H_1$ can be expressed as $h_1(.)=\sum_{j=0}^N\alpha_j K(.,t_{j})$. From the reproducing property $\left(K(.,t_{j}),K(.,t_{i})\right)_{H}=K(t_{i},t_{j})$, we get
\begin{equation*}
\| h_1\|_{H}^2=\left(h_1,h_1\right)_{H}=\sum_{i,j=0}^N\alpha_i\alpha_jK(t_{i},t_{j})=\alpha^{\top}\Gamma_N\alpha.
\end{equation*}
As $h_1(t_{i})=\sum_{j=0}^N\alpha_j K(t_{i},t_{j})$ for $i=0,\ldots,N$, we have $\alpha=\Gamma_N^{-1}c_{h_1}$ and
\begin{equation*}
\| h_1\|_{H}^2=c_{h_1}^{\top}\Gamma_N^{-1}\Gamma_N\Gamma_N^{-1}c_{h_1}=c_{h_1}^{\top}\Gamma_N^{-1}c_{h_1}.
\end{equation*}
Since $h_0\in H_0$, $c_{h_1}=c_h$ and $\|h_1\|_H^2=c_{h}^{\top}\Gamma_N^{-1}c_{h}=\| \pi_N(h)\|_{H_N}^2$, which completes the proof  (\ref{stab}). \\
The caracterisation (\ref{newdefH}) of $H$ and the property (\ref{normH}) has been proposed by Parzen in \cite{parzen1959}. In \cite{bay1}, Theorem 3.1 p 1587, Bay, Grammont and Maatouk give a proof easier to understand in the framework of this paper. 
\end{proof}

\begin{prop} {\bf Isometric property of $\rho_N$} \\
For all $v_N \in H_N$, we have
\begin{equation}\label{normrhon}
\| \rho_N (v_N)\|_{H}^2=c_{v_N}^{\top}\Gamma_N^{-1}c_{v_N}, 
\end{equation}
 
The operator   $\rho_N$ is an isometry from $H_N$ into $H$, i.e.
\begin{equation}\label{iso}
\forall v_N \in H_N, \qquad  \| \rho_N(v_N)\|^2_{H}=\| v_N\|^2_{H_N}.
\end{equation} 

$\forall h \in H$,
\begin{equation}\label{consistance}
\Vert \rho_N(\pi_N(h))-h \Vert_H  \underset{N\to +\infty} \rightarrow 0.   
\end{equation} 
\end{prop} 
 
\begin{proof}
We have 
\begin{equation*}
\| \rho_N( v_N)\|_{H}^2=(\rho_N (v_N),\rho_N (v_N))_H= \sum_{i=0}^N \sum_{j=0}^N \alpha_j \alpha_i \left( K(.,t_{i}) K(.,t_{j})\right)_H.
\end{equation*}
Since $\left(K(.,t_{i}),K(.,t_{j})\right)_H =K(t_{i},t_{j})$, we obtain
\begin{equation*}
\| \rho_N (v_N)\|_{H}^2= \sum_{i=0}^N \sum_{j=0}^N \alpha_j \alpha_i K(t_{N,i},t_{N,j})= \alpha^{{\top}}\Gamma_N \alpha.
\end{equation*}
As $\alpha=\Gamma_N^{-1}c_{h_N}$ and $\Gamma_N$ is symmetric, we obtain (\ref{normrhon}). We have  $v_N=c_{v_N}^{\top}\varphi(x)$. According to the definition of the inner product in $H_N$, we have  
\begin{eqnarray*}
\| v_N\|^2_{H_N}=\left(  v_N  , v_N  \right)_{H_N}=c_{v_N}^{\top}\Gamma_N^{-1} c_{v_N}. 
\end{eqnarray*} 
Using (\ref{normrhon}), we obtain $\| \rho_N(v_N)\|^2_{H}=\| v_N\|^2_{H_N}$. \\
We have 
\begin{eqnarray*}
\Vert \rho_N(\pi_N(h))-h \Vert_H &=& (\rho_N(\pi_N(h))-h,\rho_N(\pi_N(h))-h)_H \\
&=& \| \rho_N(\pi_N(h))\|^2_{H} + \|h\|^2_{H} -2(h, \rho_N(\pi_N(h)) \\
\end{eqnarray*}
  Thanks to (\ref{iso}), and according to (\ref{QN}) with 
  $\Lambda=\Gamma_N^{-1}c_{h}.$
\begin{eqnarray*}
\Vert \rho_N(\pi_N(h))-h \Vert_H &=&  
  \|  \pi_N(h)\|^2_{H_N} + \|h\|^2_{H} -2(h, \displaystyle\sum_{i=0}^{N} \lambda_i K(.,t_i)  ) \\
  &=&  
  \|  \pi_N(h)\|^2_{H_N} + \|h\|^2_{H} -2 \displaystyle\sum_{i=0}^{N} \lambda_i h(t_i)   \\
  &=&  
  c_{h}^{\top}\Gamma_N^{-1} c_{h} + \|h\|^2_{H} -2c_{h}^{\top}\Gamma_N^{-1} c_{h}   \\
  &=&  
    \|h\|^2_{H} - c_{h}^{\top}\Gamma_N^{-1} c_{h}   \\
     &=&  
    \|h\|^2_{H} -  \|  \pi_N(h)\|^2_{H_N}   \\
\end{eqnarray*}
And (\ref{consistance}) comes from (\ref{normH}).

\end{proof}

Now, we can formulate the approximation problem~:\\
Let us define the function $J_N:H_N\rightarrow\RR^+$: $\forall v_N \in H_N$, 
\begin{equation}\label{JN}
J_N(v_N):=\Vert v_N\Vert^2_{H_N}+ \frac{1}{\sigma^2}\Vert {\cal I}_n(v_N)- \boldsymbol{y}^{\top} \Vert^2_n
\end{equation}  
where  $ \boldsymbol{y}=(y_1,\ldots,y_n)$   and $\Vert .\Vert_n$ the euclidian norm in $\RR^n$. 

\begin{prop}\label{propJN} $\forall u_N \in H_N, \ \forall v_N \in H_N$, we have
\begin{equation}\label{egalJ_N}
J_N(tu_N+(1-t)v_N)=tJ_N(u_N)+(1-t)J_N(v_N) -t(1-t)\left(\Vert u_N-v_N\Vert^2_{H_N}+\frac{1}{\sigma^2}\Vert {\cal I}_n(u_N)-{\cal I}_n(u_N)\Vert^2_n\right).
\end{equation}
So that $J_N$ is strongly convex~:
\begin{equation}\label{fortconv}
J_N(tu_N+(1-t)v_N)\leq tJ_N(u_N)+(1-t)J_N(v_N) -t(1-t)\Vert u_N-v_N\Vert_{H_N}.
\end{equation}
Moreover $J_N$ is Fr\'echet differentiable and 
$$\displaystyle\lim_{\Vert v_N\Vert_{H_N} \mapsto+\infty} J_N(v_N)=+\infty $$
Moreover 
\begin{equation}\label{JJ_N}
J_N(v_N)=J(\rho_N(v_N)).
\end{equation}
\end{prop}

\begin{proof}

As $2(u_N,v_N)_{H_N}=\Vert u_N\Vert^2_{H_N}+\Vert v_N\Vert^2_{H_N}-\Vert u_N-v_N\Vert^2_{H_N}$, we have
\begin{eqnarray*}
J_N(tu_N+(1-t)v_N) &=& t^2\Vert u_N\Vert^2_{H_N}+(1-t)^2\Vert v_N\Vert^2_{H_N}+2t(1-t)(u_N,v_N)_{H_N}
 +\frac{1}{\sigma^2}\Vert {\cal I}_n(tu_N+(1-t)v_N)- \boldsymbol{y}^{\top} \Vert^2_n \\
 &=& t\Vert u_N\Vert^2_{H_N}+(1-t)\Vert v_N\Vert^2_{H_N}-t(1-t)\Vert u_N-v_N\Vert^2_{H_N}+
 \frac{1}{\sigma^2}\Vert {\cal I}_n(tu_N+(1-t)v_N)- \boldsymbol{y}^{\top} \Vert^2_n \\
\end{eqnarray*}
As $2 ({\cal I}_n(u_N)-\boldsymbol{y}^{\top} , {\cal I}_n(v_N)-\boldsymbol{y}^{\top})_n= \Vert {\cal I}_n(u_N)-\boldsymbol{y}^{\top}\Vert^2_n + \Vert {\cal I}_n(v_N)-\boldsymbol{y}^{\top}\Vert^2_n-\Vert {\cal I}_n(u_N)-{\cal I}_n(u_N)\Vert^2_n $
\begin{eqnarray*}
\frac{1}{\sigma^2}\Vert {\cal I}_n(tu_N+(1-t)v_N)- \boldsymbol{y}^{\top} \Vert^2_n &=& \frac{1}{\sigma^2}t^2\Vert {\cal I}_n(u_N)- \boldsymbol{y}^{\top} \Vert^2_n +
\frac{1}{\sigma^2}(1-t)^2\Vert {\cal I}_n(v_N)- \boldsymbol{y}^{\top} \Vert^2_n \\
&+& 2t(1-t)({\cal I}_n(u_N)-\boldsymbol{y}^{\top} , {\cal I}_n(v_N)-\boldsymbol{y}^{\top})_n\\
&=& t\frac{1}{\sigma^2}\Vert {\cal I}_n(u_N)- \boldsymbol{y}^{\top} \Vert^2_n +
(1-t)\frac{1}{\sigma^2}\Vert {\cal I}_n(v_N)- \boldsymbol{y}^{\top} \Vert^2_n \\
&-& t(1-t)\frac{1}{\sigma^2}\Vert {\cal I}_n(u_N)-{\cal I}_n(u_N)\Vert^2_n\\
\end{eqnarray*}
So that 
\begin{eqnarray*}
J_N(tu_N+(1-t)v_N) &=&   tJ_N(u_N)+(1-t)J_N(v_N)-t(1-t)  \Vert u_N-v_N\Vert^2_{H_N}\\
&-& t(1-t)\frac{1}{\sigma^2}\Vert {\cal I}_n(u_N)-{\cal I}_n(u_N)\Vert^2_n
\end{eqnarray*}
It is easy to notice that ${\cal I}_N(\rho_N(v_N))={\cal I}_N(v_N)$, so that, as  (\ref{H3}) is satisfied, then 
$${\cal I}_n(\rho_N(v_N))={\cal I}_n(v_N)$$
so that, thanks to (\ref{iso})
\begin{eqnarray*}
J(\rho_N(v_N)) &=&     \Vert \rho_N(v_N) \Vert^2_{H }+\frac{1}{\sigma^2}\Vert {\cal I}_n(\rho_N(v_N))- \boldsymbol{y}^{\top} \Vert^2_n \\
&=& \Vert  v_N \Vert^2_{H_N} +\frac{1}{\sigma^2}\Vert {\cal I}_n( v_N)- \boldsymbol{y}^{\top} \Vert^2_n \\
&=& J_N(v_N)
\end{eqnarray*}
 
\end{proof}

The constraints space is simply defined as  
\begin{equation}\label{CN}
C_N=H_N\cap C
\end{equation}

Then the discretized problem is defined by
\begin{equation}\label{pbdis}\tag{$P_N$}
\min_{u_N\in  C_N } J_N(u_N)
\end{equation}
\begin{theorem}
Under the   assumptions (\ref{H1}) to (\ref{H5}),\\
   (\ref{pbdis})   has a unique solution $\widehat{u}_N$.
\end{theorem} 
 \begin{proof}
 Let $g\in H\cap C $, then, thanks to hypothesis (\ref{H4}), 
 $\pi_N(g) \in C_N=H_N \cap C$. So that $C_N$ is a nonempty closed convex of $H_N$. According to the properties of $J_N$ (Proposition \ref{propJN}), we have the conclusion. 
 \end{proof}

\section{Convergence result}
The aim of this   paragraphe is to prove that, if $\widehat{u}_N$ is the solution of (\ref{pbdis}) and $\widehat{u}  $ the solution of (\ref{prob}),  then 
$$ \lim\limits_{N \rightarrow +\infty} \widehat{u}_N =\widehat{u}    $$ 
 
 We will prove two intermediate results leading to the convergence result. 
 \begin{prop}\label{convJ}
 Under (\ref{H1}) (\ref{H2}) and (\ref{H3})
 \begin{eqnarray}\label{int1}
& \lim\limits_{N \rightarrow +\infty}& J_N(\pi_N(\widehat{u}  ))=J(\widehat{u}  )\\
& \lim\limits_{N \rightarrow +\infty}&J_N( \widehat{u}_N)=J(\widehat{u}  )
\end{eqnarray}
 \end{prop}
\begin{proof}
Let us set $$h^N:=\rho_N(\widehat{u}_N)\in H.$$
Using (\ref{JJ_N}), as $\pi_N(\widehat{u}  )\in H_N\cap C$, we have 
\begin{equation}\label{ine}
J(h^N)=J(\rho_N(\widehat{u}_N))=J_N(\widehat{u}_N) \leq J_N(\pi_N(\widehat{u}  ))\leq J(\widehat{u}  )
\end{equation}
Let us prove the last inequality:
 $J_N(\pi_N(\widehat{u}  ))=\Vert\pi_N(\widehat{u}  )\Vert_{H_N}+\displaystyle\frac{1}{\sigma^2}\Vert  {\cal I}_n(\pi_N( \widehat{u}  ))- y \Vert^2_n.  $\\
  Thanks to (\ref{H3}),  ${\cal I}_n(\pi_N(\widehat{u}  ))={\cal I}_n(\widehat{u}  )$ so that, thanks to (\ref{stab}),    $$J_N(\pi_N(\widehat{u}  ))\leq \Vert\ \widehat{u}   \Vert_{H}+\displaystyle\frac{1}{\sigma^2}\Vert {\cal I}_n(\widehat{u}  )- y \Vert^2_n=J(\widehat{u}  ) $$ 
As $\Vert h^N\Vert_{H} \leq J(h^N) \leq J(\widehat{u}  ) $, then the sequence $(h^N)_{N\in \NN}$ is bounded in $H$ so that, by weak compactness in Hilbert space,  there exists a sub-sequence $(h^{N_k})_{k \in \NN}$ and  $h^* \in H$ such that    
\begin{equation}\label{weakcompact}
 h^{N_k}\underset{k\to +\infty}\rightharpoonup h^*\in H, \qquad \text{(weak convergence)}.
\end{equation}
  
As $H$ is a RKHS with kernel $K$, for all $t_i \in  \Delta_N, \quad K_(.,t_i)\in H$ and
 $$(h^{N_k},K(.,t_i))_H=h^{N_k}(t_i) \underset{k\to +\infty}\rightharpoonup
(h^{*},K(.,t_i))_H=h^{*}({t_i}).$$

Therefore for all $N \geq 1$,  $\pi_N(h^{N_k})\underset{k\to +\infty}\rightarrow \pi_N(h^*) $ in the finite dimensional space $H_N$.\\
 According to   assumption (\ref{H2}),  as far as $N_k\geq N$,  ${\cal I}_{N}(h^{N_k})={\cal I}_{N}(\rho_{N_k}(\widehat{u}_{N_k} ))={\cal I}_{N}(\widehat{u}_{N_k} )$ and
$$\pi_{N}(h^{N_k})=\pi_{N}(\rho_{N_k}(\widehat{u}_{N_k} ))=\pi_{N}(\widehat{u}_{N_k} ),$$
so that 
$$\pi_{N}(\widehat{u}_{N_k} ) \underset{k\to +\infty}\rightarrow \pi_N(h^*) \quad \mbox{in} \ H_N$$
 
As $H_N$ is an Hibertian subspace of $E$ ( inequality (\ref{HN}) of   proposition \ref{rkhsHN}.), 
$$\pi_{N}(\widehat{u}_{N_k} ) \underset{k\to +\infty}\rightarrow \pi_N(h^*) \quad \mbox{in} \  E$$

Under  (\ref{H4}), $\pi_N(\widehat{u}_{N_k}  ) \in  C$  and $C$ is closed in E, so that $\forall N$,  $$\pi_N(h^*) \in   C.$$
$C$ is closed in $E$ and  $\pi_N(h^*) \underset{N\to +\infty}\rightarrow h^*$ in $E$, then 
$$h^* \in C \quad \mbox{and} \quad J(\widehat{u}  ) \leq J(h^*) $$
Then, as $J$ is convex and lower semi continuous and $h^{N_k}\underset{k\to +\infty}\rightharpoonup h^*\in H$, using (\ref{ine}),
$$J(\widehat{u}  ) \leq J(h^*)\leq \liminf_k J(h^{N_k})=\liminf_k J_{N_k}(\widehat{u}_{N_k} ) \leq \limsup_k J_{N_k}(\widehat{u}_{N_k} ) \leq J(\widehat{u}  ) $$ so that 
$$ J_{N_k}(\widehat{u}_{N_k} ) \underset{k\to +\infty}\rightarrow  J(\widehat{u}  ) $$

(\ref{ine}) implies that 
\begin{equation*} 
J(\widehat{u}  ) \leq J(h^*)\leq \liminf_k J(h^{N_k})=\liminf_k J_{N_k}(\widehat{u}_{N_k} ) \leq  \liminf_k J_{N_k}(\pi_{N_k}(\widehat{u}  )) \leq  \limsup_k J_{N_k}(\pi_{N_k}(\widehat{u}  )) \leq J(\widehat{u}  )
\end{equation*}

so that 
$$ J_{N_k}(\pi_{N_k}\widehat{u}  ) \underset{k\to +\infty}\rightarrow  J(\widehat{u}  ) $$
As the sequences $(J_{N}(\widehat{u}_{N}))_{N \in \NN}$ and $(J_{N}(\pi_{N}\widehat{u}  ))_{N \in \NN}$ are in the compact set $[0,J(\widehat{u}  )]$, then the results hold. 
\end{proof}

\begin{prop}\label{intermed}
 \begin{eqnarray}\label{int2}
\lim\limits_{N \rightarrow +\infty}  \Vert \pi_N(\widehat{u}  ))-\widehat{u}_N\Vert_{H_N} =0
\end{eqnarray}
 \end{prop}
 
\begin{proof}
As $J_N$ is strongly convex (\ref{fortconv}) and differentiable, then 
 $$J_{N}(\pi_N(\widehat{u}  ))-J_{N}(\widehat{u}_{N} ) \geq (J_{N}'(\widehat{u}_{N}),\pi_{N}(\widehat{u}  )-\widehat{u}_{N} )_{H_N}+\Vert \pi_N(\widehat{u}  ))-\widehat{u}_N\Vert^2_{H_N},  $$ 
 where $J_{N}'$ denotes the derivative of $J_N$. 
 As $\pi_N(\widehat{u}  ) \in H_N\cap C$ and $\widehat{u}_{N}$ solves (\ref{pbdis}), $$(J_{N}'(\widehat{u}_{N}),\pi_{N}(\widehat{u}  )-\widehat{u}_{N} )_{H_N} \geq 0,$$
 so that 
 $$   \Vert \pi_N(\widehat{u}  ))-\widehat{u}_N\Vert^2_{H_N} \leq  J_{N}(\pi_N(\widehat{u}  ))-J_{N}(\widehat{u}_{N})  $$ 
 (\ref{int2}) comes from the application of proposition \ref{convJ}.
\end{proof}

\begin{theorem}
Under (\ref{H1}), (\ref{H2}), (\ref{H3}) (\ref{H4}) and (\ref{H5})
\begin{equation*}
\widehat{u}_N  \underset{N\to +\infty}{\longrightarrow}\widehat{u}   \quad \mbox{in $E$},
\end{equation*}
\end{theorem}

\begin{proof}
\begin{equation*}
\|\widehat{u}_N -\widehat{u}  \|_E \leq \|\widehat{u}_N -\pi_
N(\widehat{u}  )\|_E+\|\pi_N(\widehat{u}  )-\widehat{u}  \|_E.
\end{equation*}
 We know from approximation theory in the Banach $E$ that 
\begin{equation*}
\|\ \pi_N(\widehat{u}  )-\widehat{u}  \|_E \underset{N\to +\infty}\longrightarrow 0.
\end{equation*}
As $H_N$ is an Hilbertian subspace of $E$ (see (\ref{HN})),   
$$\|\ \pi_N(\widehat{u}  )-\widehat{u}  \|_E  \leq c \|\ \pi_N(\widehat{u}  )-\widehat{u}  \|_{H_N}.$$
Proposition \ref{intermed} gives the result. 
\end{proof}

\section{Stochastic Correspondence of Constrained Optimal Smoothing}
The overall goal of the paper  is to find a correspondance between the solution $\widehat{u}  $ of (\ref{prob}) and the posterior distribution  $\{ U(x)\suchthat U \in C,Y_i=y_i \}$,
  $(U_x)_{x \in [0,1]}$ being the Gaussian process associated to the covariance fonction $K$, the kernel of the RKHS $H$.  The observations are written as    $$Y_i=U(x_i)+E_i, 1\leq i \leq n,$$ with  $E=(E_i)_i$ is a centered Gaussian vector ${\cal N}(0,\sigma I)$.\\

As we defined a finite dimensional approximation space $H_N$ of $H$, it is quite natural to construct a finite dimensional Gaussian process $U_N$ to approach $U$. 
Using the subdivision (\ref{partition}),  we approximate the Gaussian process  $U$ by the following finite-dimensional Gaussian process~: 
\begin{equation}\label{YN}
U_N(x):=\sum_{j=0}^{N}U(t_{j})\varphi_{j}(x), \qquad x\in X,
\end{equation}
   Note that $\xi := \left(U(t_{0}), \ldots,U(t_{N})\right)^\top$ is a zero-mean Gaussian vector with covariance matrix $\Gamma_N=\left(K(t_{i},t_{j})\right)_{0\leq i,j\leq N}$, where $K$ is the covariance function of  $U$.  \\
   In the following proposition, we prove that the Gaussian process $U_N$ is associated with the Hilbertian space $H_N$ defined in section 3.

\begin{prop}\label{prop13}
The process $U_N$ defined by (\ref{YN}) is a zero-mean Gaussian process with covariance function~: 
\begin{equation}\label{KN}
\forall x', x\in [0,1], \qquad K_N(x',x)= \sum_{i,j=0}^{N} K(t_{i},t_{j}) \varphi_{j}(x)\varphi_{i}(x')=\phi(x)^{{\top}}\Gamma_N\phi(x').
\end{equation}
 
where $\phi=(\varphi_{0},\ldots,\varphi_{N})^\top.$  
If  (\ref{H5}) is satisfied, then the RKHS associated to $U_N$ with the reproducing kernel $K_N$ is 
\begin{equation*} 
H_N:=\vect\{\varphi_{j}, \ j=0,\ldots,N\}
\end{equation*} 
with the scalar product (\ref{proscalHN})
\end{prop}

\begin{proof}
 Obvious
\end{proof}

In the following proposition, we denote by $\overset{\circ}{\widehat{C_N}}$ the interior of $C_N$ in the finite-dimensional space $H_N$, where $C_N$ is defined by (\ref{CN}).

\begin{theorem}\label{thm3}
The posterior likelihood function of $\{U_N\suchthat U_N\in C,  Y_i=y_i \}$   is of the form
\begin{equation}\label{finitePostDis}
L_{pos}^N(h)=k_N^{-1}\mathds{1}_{h\in   C_N\cap I}\exp\left(-\frac{1}{2}J_N(h)\right),
\end{equation}
where $k_N$ is a normalizing constant. Then, the MAP estimator $\widehat{v}_N$ as the mode of the posterior distribution   $\{U_N\suchthat U_N\in C,  Y_i=y_i \}$   is well defined and is equal to   $\widehat{u}_N $ solution of 
(\ref{pbdis}).
\end{theorem}

\begin{proof}
 First, remark that the sample paths of $U_N$ are in $H_N$ by construction. Hence, it makes sense to define the density of $U_N$ with respect to the uniform reference measure $\lambda_N$ on $H_N$. The density is defined up to a multiplicative constant and to give it an explicit expression, we consider the following linear isomorphism~: 
\begin{equation*}\label{iso}
i~:c\in\mathbb{R}^{N+1} \longmapsto u:=\sum_{j=0}^N c_j\varphi_j\in H_N.
\end{equation*}
We can define the measure $\lambda_N$ on $H_N$ as the image measure $\lambda_N:=i(dc)$, where $dc=dc_0\times\ldots \times dc_N$ is the   volume measure in $\mathbb{R}^{N+1}$. So, if $B\in \mathcal{B}(H_N)$ is a Borelian subset of $H_N$, we have
\begin{equation*}
\lambda_N(B)=\int_{\mathbb{R}^N}\mathds{1}_{i^{-1}(B)}(c)dc_1\times\ldots\times dc_N.
\end{equation*} 
To calculate the probability density function  of the Gaussian Process $U_N=\sum_{j=0}^{N}U(t_{j})\varphi_{j}$, we write 
\begin{equation*}
P(U_N\in B)=P\left(\xi\in i^{-1}(B)\right).
\end{equation*} 
Using the fact that $\xi= \left(U(t_{0}), \ldots,U(t_{N})\right)^\top$ is a  zero-mean Gaussian vector $\mathcal{N}(0,\Gamma_N)$, we have
\begin{eqnarray*}
P(U_N\in B)&=&\int_{\mathbb{R}^N}\mathds{1}_{i^{-1}(B)}(c)\frac{1}{\sqrt{2\pi}^N|\Gamma_N|^{1/2}}\exp{\left(-\frac{1}{2}c^\top\Gamma_N^{-1} c\right)}dc\\
&=&\int_{\mathbb{R}^N}\mathds{1}_{B}(i(c))\frac{1}{\sqrt{2\pi}^N|\Gamma_N|^{1/2}}\exp{\left(-\frac{1}{2}\|i(c)\|_N^2\right)}dc.
\end{eqnarray*}
By the transfer formula, we get 
\begin{equation*}
P(U_N\in B)=\int_{H_N}\mathds{1}_B(u)\frac{1}{\sqrt{2\pi}^N|\Gamma_N|^{1/2}}\exp{\left(-\frac{1}{2}\|u\|_N^2\right)}d\lambda_N(u).
\end{equation*} 
Hence, the  density of $U_N$ with respect to $\lambda_N$
is the function 
\begin{equation*}
f_{\{U_N\}}:u\in H_N\longmapsto \frac{1}{\sqrt{2\pi}^N|\Gamma_N|^{1/2}}\exp{\left(-\frac{1}{2}\|u\|_N^2\right)}.
\end{equation*}

Now let us find the density of the conditional distribution $\{U_N\suchthat   Y_i=y_i \}$. \\

As $Y_i=U_N(x_i)+{\cal E}_i$, where ${\cal E}=({\cal E}_i)_i$ is a centered Gaussian vector ${\cal N}(0,\sigma I)$, the density of  $\{  (Y_1,\ldots,Y_n)  \suchthat U_N =u \}$ is given by 
\begin{equation*}
f_{\{(Y_1,\ldots,Y_n)\suchthat U_N =u\}}:(y_1,\ldots,y_n)\longmapsto \frac{1}{\sqrt{2\pi}^n}\exp\left({-\frac{1}{2\sigma^2}\sum_{i=1}^n(y_i-u(x_i))^2}\right).
\end{equation*}
Let us apply the Bayes principle: the density of the distribution $\{U_N\suchthat   Y_i=y_i \}$ is given by
\begin{equation*}
f_{\{U_N\suchthat   Y_i=y_i\}}(u)=\displaystyle\frac{f_{\{(Y_1,\ldots,Y_n)\suchthat U_N =u\}}(y_1,\ldots,y_n) \times f_{\{U_N\}}(u)}{f_{\{(Y_1,\ldots,Y_n)\}} (y_1,\ldots,y_n)}.
\end{equation*}
We consider $f_{\{(Y_1,\ldots,Y_n)\}} (y_1,\ldots,y_n)$ as a constant of normalization $k$. Then 
\begin{equation*}
f_{\{U_N\suchthat   Y_i=y_i\}}(u)=k^{-1} \frac{1}{\sqrt{2\pi}^n}\frac{1}{\sqrt{2\pi}^N|\Gamma_N|^{1/2}} \exp\left({-\frac{1}{2\sigma^2}\sum_{i=1}^n(y_i-u(x_i))^2}-\frac{1}{2}\|u\|_N^2\right).
\end{equation*}
Setting $k_{N,n}=k^{-1}\displaystyle \frac{1}{\sqrt{2\pi}^n}\frac{1}{\sqrt{2\pi}^N|\Gamma_N|^{1/2}}$, 
\begin{equation}\label{postdistrib}
f_{ \{U_N\suchthat   Y_i=y_i\}}(u)=k_{N,n}\exp\left( -\frac{1}{2}J_N(u)\right).
\end{equation}

Let us  introduce the inequality constraints described by the set $C$. if $B\in \mathcal{B}(H_N)$,
\begin{equation*}
P( U_N \in B \suchthat U_N \in C,   Y_i=y_i) = P_{Y_i=y_i} ( U_N \in B \suchthat U_N \in C)=\frac{ P_{Y_i=y_i}( U_N \in B\cap C) }{ P_{Y_i=y_i}( U_N \in  C)},
\end{equation*} 
so that 

\begin{equation*}
P( U_N \in B \suchthat U_N \in C,   Y_i=y_i) =  \displaystyle\frac{1}{ P_{Y_i=y_i}( U_N \in   C)} \int_{H_N}\mathds{1}_{B\cap C}(u)f_{ \{U_N\suchthat   Y_i=y_i\}}(u) d\lambda_N(u) \Leftrightarrow  
\end{equation*} 
\begin{equation*}
P( U_N \in B \suchthat U_N \in C,   Y_i=y_i) =  \displaystyle\frac{k_{N,n}}{ P_{Y_i=y_i}( U_N \in   C)} \int_{H_N\cap C}\mathds{1}_{B}(u) \exp\left( -\frac{1}{2}J_N(u)\right) d\lambda_N(u).  
\end{equation*} 
In the Bayesian framework, as far as $P_{Y_i=y_i}( U_N \in   C)\neq 0$, the  density of the posterior conditional distribution  $\{ U_N  \suchthat U_N \in C,   Y_i=y_i\}$ is the following truncated probability density function  with respect to $\lambda_N$~:
\begin{equation}\label{condpostdist}
f_{\{U_N \in B \suchthat U_N \in C,   Y_i=y_i\}}(u)=c \  \mathds{1}_{\{u\in H_N\cap C\}}\exp{\left(-\frac{1}{2} J_N(u)\right)}.
\end{equation} 
 The density $f_{\{U_N \in B \suchthat U_N \in C,   Y_i=y_i\}}$ is also called the posterior likelihood function denoted by $L^N_{pos}$. By definition, the MAP estimator $\hat{u}_N$ is the solution of the following optimization problem
\begin{equation*}
\arg\max L^N_{pos}(u)=\arg\min\left(-2 \log L^N_{pos}(u)\right).
\end{equation*}   
From expression (\ref{condpostdist}), the MAP estimate $\widehat{v}_N$ is   $\widehat{u}_N $ solution of 
(\ref{pbdis}) 
\end{proof}

\section{Conclusion} 
In this paper, we considered the constrained optimal problem. We approach it by the usual piecewise linear projection to obtain an approximate solution. We proved the convergence of these approximations to the optimal solution. If we rewrite the problem in a certain model uncertainty framework, we realized that the  obtained approximate solution is also the Maximum A Posteriori of the posterior distribution of an approximation of the stochastic model. These results are theoretical but they can lead to practical application: this paper gives two options: deterministic optimisation algorithms or methods of Bayesian estimation.

\end{document}